\documentclass{scrartcl}
\usepackage[utf8]{inputenc}
\usepackage[T1]{fontenc}
\usepackage[english]{babel}

\usepackage{amsmath}
\usepackage{amssymb}
\usepackage{amsthm}

\usepackage{enumerate}
\usepackage{mdframed}

\usepackage{todonotes}
\usepackage{hyperref}

\newtheorem{proposition}{Proposition}
\newtheorem{theorem}{Theorem}
\newtheorem{lemma}{Lemma}

\theoremstyle{definition}
\newtheorem{definition}{Definition}

\theoremstyle{remark}
\newtheorem{remark}{Remark}

\DeclareMathOperator{\lip}{Lip}
\DeclareMathOperator{\diam}{diam}
\DeclareMathOperator{\fix}{Fix}
\DeclareMathOperator{\argmin}{argmin}

\title{Generic Convergence of\\Sequences of Successive Approximations\\in Banach~Spaces}
\subtitle{Variations on a Theme of Francesco S.~de~Blasi}
\author{Christian Bargetz\footnote{Department of Mathematics, University of Innsbruck, Technikerstraße 13, 6020 Innsbruck, Austria, \texttt{christian.bargetz@uibk.ac.at}}\and Simeon Reich\footnote{Department of Mathematics, The Technion---Israel Institute of Technology,
32000 Haifa, Israel, \texttt{sreich@technion.ac.il}}}

\begin{document}
\maketitle
\begin{abstract}
  \noindent\textbf{\textsf{Abstract.}} We study the generic behavior of the method of successive approximations for set-valued mappings in Banach spaces. We consider, in particular, the case of those set-valued mappings which are defined by pairs of nonexpansive mappings and give a positive answer to a question raised by Francesco S.~de~Blasi.
  \vskip2mm
  \noindent\textbf{\textsf{Keywords.}} Banach space, generic property, set-valued nonexpansive mapping, successive approximations.
  \vskip2mm
  \noindent\textbf{\textsf{Mathematics Subject Classifications.}} 47H04, 47H09, 47H10, 54E52
\end{abstract}

\section{Introduction}

In view of the absence of Brouwer's fixed point theorem for bounded, closed and convex subsets of infinite-dimensional Banach spaces, there has been considerable interest in studying generic existence of fixed points of nonexpansive mappings defined on such sets. In~\cite{DM1976Convergence} F.~S.~de~Blasi and J.~Myjak show that a generic nonexpansive self-mapping $f$ of a bounded, closed and convex subset of a Banach space has a fixed point and that this fixed point can be found by first choosing an arbitrary starting point $x_0$ and then using the iterative steps $x_{k+1} = f(x_k)$, where $k = 0, 1, 2, \ldots$. 
Since in the case of Hilbert spaces, it is also shown in~\cite{DM1976Convergence} that a typical nonexpansive mapping has the maximal possible Lipschitz constant, namely $1$, the question of which fixed point result is behind this generic convergence result also arose. 
The result regarding the smallness of the set of strict contractions (that is, of those mappings the Lipschitz constant of which is strictly less than $1$) was later generalized to Banach spaces in~\cite{MR3540613} and to more general settings, including some classes of set-valued mappings, in~\cite{MR3706153}. An answer to the question of what is behind the generic convergence was given by A.~Zaslavski and the second author in~\cite{MR1860926}, where they show that a typical nonexpansive mapping $f$ on a bounded, closed and convex subset $D$ of a Banach space is contractive in the sense of Rakotch \cite{MR0148046}, that is, there is a decreasing function $\phi\colon (0,\infty)\to[0,1)$ such that
\[
  \|f(x)-f(y)\| \leq \phi(\|x-y\|)\|x-y\|
\]
for all $x,y\in D$.
This result has been generalized in~\cite{DM1976Convergence} and~\cite{MR3375974} to set-valued mappings the point images of which are closed subsets of their domains of definition. 
Recall that given a bounded, closed and convex subset $D$ of a Banach space $X$ and a mapping
\[
  F\colon D\to 2^{D}\setminus\{\emptyset\},
\]
an element $x\in D$ is called a fixed point of $F$ if it belongs to its image under $F$, that is, if it satisfies $x\in F(x)$. The proofs of these results use iterations of a mapping defined on suitable hyperspaces. More precisely, let $\mathcal{B}(D)$ be a hyperspace of certain nonempty and closed subsets of $D$, and let $F\colon D\to\mathcal{B}(D)$ be a nonexpansive mapping. Then the mapping
\[
  \tilde{F}\colon \mathcal{B}(D)\to\mathcal{B}(D), \qquad A \mapsto \bigcup_{x\in A} F(x),
\]
is considered. Thus, in some sense the problem is ``lifted'' from a set-valued mapping to a single-valued self-mapping of a certain hyperspace. A different approach to this problem is taken in ~\cite{deBlasi} 
and~\cite{Pia2015GenericProperties}. In the particular case of set-valued mappings which are defined by pairs of nonexpansive single-valued mappings, an iteration of the form
\begin{equation}\label{eq:Sequence}
  x_{0}\in D, \qquad x_{k+1}\in \argmin\{\|y-x_k\|\colon y\in F(x_k)\}, 
\quad k\in\mathbb{N},
\end{equation}
is considered. 
In the case where $D$ is a bounded, closed and convex subset of a Hilbert space, the authors consider these mappings as elements of the space
\[
  \mathcal{M} :=  \{\{f,g\}\colon f,g\colon D\to D\;\text{with}\;
\lip f \leq 1\},
\]
which is equipped with the Hausdorff distance inherited from the space of nonexpansive self-mappings endowed with the metric of uniform convergence. Here and in the sequel we denote by $\lip f$ the Lipschitz 
constant of a mapping $f$. It is shown there that given a fixed initial point $x_0$, the set of those nonexpansive mappings for which the sequence in~\eqref{eq:Sequence} is unique and converges is a residual subset of $\mathcal{M}$. In~\cite{deBlasi}, F.~S.~de~Blasi raises the question of whether this result is still true in general Banach spaces. The main difference between the Banach space case and the Hilbert space case is that in the former the Kirszbraun-Valentine extension theorem, on which the proofs in~\cite{deBlasi,Pia2015GenericProperties} are based, is no longer available. One of the goals of the present paper is to give a positive answer to this question. We remark in passing that a detailed overview of the genericity approach and its applications to nonlinear analysis can be found in the book~\cite{MR3137655}.

Since the set-valued mappings considered in this paper are of the form $\{f,g\}$, where both $f,g\colon D\to D$ are nonexpansive (single-valued) mappings, there seem to be at least two natural metrics on the space $\mathcal{M}$: the Hausdorff distance and the metric of uniform convergence. The former was used in~\cite{deBlasi,Pia2015GenericProperties}, whereas the latter is the one used in the results on generic existence of fixed points for set-valued mappings mentioned above. We compare these two metrics and show that for both of them, the results of~\cite{deBlasi,Pia2015GenericProperties} can be 
generalized to the Banach space setting. We emphasize that our goal in the present paper is not to find a fixed point of the mapping $f$ or $g$, which could be done by considering both of them independently, but to work towards an understanding of the generic behavior of the method of successive approximations for set-valued mappings.

Mappings of the form
\[
  X \to 2^{X}\setminus\emptyset, \qquad x \mapsto \{f_i(x)\colon i\in 
I(x)\},
\]
where each $f_i$ is a nonexpansive mapping and $I(x)\subset I$, where $I$ is a finite index set, have also been recently considered by M.~Tam and others in the context of local convergence analysis of 
iterative methods; see, for example~\cite{MR3819673, 2018arXiv180705810D}.

In the sequel $X$ is always a Banach space and $D\subset X$ is a bounded, closed and convex set. We moreover assume that $D$ contains more than one element. By $h$ we denote the Hausdorff distance on $2^{D}\setminus \{\emptyset\}$.

\section{Some auxiliary results for single-valued mappings}
In order to analyze the set-valued case, we first need a number of auxiliary results regarding the single-valued case. Some of these results might be of independent interest.

\begin{lemma}\label{lem:DiffFixedPoint}
  Let $f\colon D\to D$ be a non-constant, nonexpansive mapping with a fixed point $\xi\in D$. Then for each $\varepsilon>0$, there is a strict contraction $\varphi\colon D \to D$ which is $\varepsilon$-close to~$f$, but has a different fixed point. Moreover, this strict contraction satisfies $\lip \varphi<\lip f$.
\end{lemma}

\begin{proof}
  Given $\varepsilon>0$, choose $y\in D$ with $0<\|\xi-y\|<\varepsilon$ and $\delta\in (0,1)$ with  $\delta < \frac{\varepsilon}{\diam D}$. We define $\varphi\colon D\to D$ by
  \[
    \varphi(x) := \delta y + (1-\delta) f(x), \qquad x\in D,
  \]
  and observe that
  \[
    \lip\varphi = (1-\delta) \lip f\qquad\text{and}\qquad \|f(x)-\varphi(x)\| = \delta \|y-f(x)\|< \varepsilon.
  \]
  Moreover, since 
  \[
    \|\varphi(\xi)-\xi\| = \delta\|\xi-y\|>0,
  \]
  it is clear that $\xi$ is not a fixed point of $\varphi$. Since according to Banach's fixed point theorem, $\varphi$ has a unique fixed point, it follows that $\varphi$ has a fixed point $\eta\neq \xi$, as 
asserted.
\end{proof}

We denote by
\[
  \mathcal{M}^{1} := \{f\colon D \to D\colon \lip f\leq 1\}
\]
the space of nonexpansive self-mappings of $D$ equipped with the metric of uniform convergence, 
\[
  d_\infty(f,g) = \sup_{x\in D}\|f(x)-g(x)\|,
\]
which turns $\mathcal{M}^{1}$ into a complete metric space.

\begin{lemma}\label{lem:PropD}
  Let $f\colon D\to D$ be nonexpansive. Then for each $\varepsilon>0$, there are a number $\theta_0>1$ and a nonexpansive mapping $\varphi\colon D\to D$ such that $d_\infty(f, \varphi) <\varepsilon$ and
  \begin{equation}
    \theta \varphi(x) + (1-\theta) x\in D
  \end{equation}
  for every $\theta<\theta_0$ and all $x\in D$. If $f$ is a strict contraction, then so is~$\varphi$.
\end{lemma}

\begin{proof}
  For each $t\in (0,1)$, define a mapping $\varphi_t\colon D\to D$ by
  \[
    \varphi_t(x) := (1-t)x+t f(x)\quad\text{for}\quad x\in D.
  \]
  Note that the convexity of $D$ guarantees that $\varphi_t$ is well defined. Moreover,
  \[
    \|\varphi_t(x)-\varphi_t(y)\|\leq (1-t) \|x-y\| + t\lip f \|x-y\| = \mu \|x-y\|
  \]
  for all $x,y\in D$, where $\mu := (1-t)+t\lip f$. Observe that $\mu\leq 1$ always holds and that $\mu<1$ whenever $\lip f <1$. For a fixed $t\in (0,1)$, we have
  \[
    x+ \frac{1}{t}(\varphi_t(x)-x) = f(x)\in D
  \]
  for all $x\in D$. Therefore we can set $\varphi := \varphi_t$ and $\theta_0 := \frac{1}{t}$.
  Finally, we choose $t$ close enough to $1$ so that $(1-t)\diam D <\varepsilon$ and observe that
  \[
    \|\varphi(x)-f(x)\| = (1-t) \|x-f(x)\| \le (1-t)\diam D < \varepsilon
  \]
  for all $x \in D$. This completes the proof of the lemma.
\end{proof}

\begin{lemma}\label{lem:perturbation}
  Let $f\colon D\to D$ be a strict contraction and let $\theta_0>1$ be so that $\theta f(x) + (1-\theta) x\in D$ for every $x\in D$ and every $\theta<\theta_0$. Moreover, let $\eta\in D$ and $\sigma\in(0,1]$. Then for every $\varepsilon>0$, there is a strict contraction $\tilde{f}\colon D\to D$ such that $d_\infty(f, \tilde{f}) < \varepsilon$,
  \begin{equation}
    \tilde{f}(x) = f(x) \;\text{for} \;x\in D\setminus B(\eta,\sigma)\quad\text{and}\quad
    \tilde{f}(\eta) = f(\eta) + c \big(f(\eta)-\eta\big)
  \end{equation}
  for some $c>0$.
\end{lemma}

\begin{proof}
  Given $\varepsilon >0$ and $\sigma>0$, we set
  \begin{equation}
    \alpha := \min\left\{\frac{1-\lip(f)}{4}, \frac{\varepsilon}{2\sigma}, \frac{\theta_0-1}{2\sigma}\right\}
  \end{equation}
  and
  \begin{equation}
    \gamma(x) := \max\{0,\sigma-\|x-\eta\|\} \min\left\{\alpha,\frac{\alpha}{\|f(\eta)-\eta\|+2\sigma}\right\}.
  \end{equation}
  The function $\gamma$ satisfies
  \[
    \|\gamma\|_\infty = \sigma\alpha \min\left\{1,\frac{1}{\|f(\eta)-\eta\|+2\sigma}\right\},\qquad
    \lip(\gamma) = \min\left\{\alpha,\frac{\alpha}{\|f(\eta)-\eta\|+2\sigma}\right\}.
  \]
  We define $\tilde{f}$ by
  \begin{equation}
    \tilde{f}(x) := f(x)+\gamma(x)(f(x)-x), \qquad x\in D.
  \end{equation}
  Note that
  \begin{align*}
    \tilde{f}(x) &= f(x)+\gamma(x)(f(x)-x)= (1+\gamma(x))f(x)-\gamma(x)x\\
                 &= \mu(x) f(x) + (1-\mu(x))x,
  \end{align*}
  where $\mu(x)=1+\gamma(x)>1$ and $\mu(x)<\theta_0$.

  The calculation
  \begin{align*}
    \big\|\tilde{f}(x)-\tilde{f}(y)\big\| &=\|f(x)+\gamma(x)(f(x)-x)-f(y)-\gamma(y)(f(y)-y))\|\\
                                  & \leq \|f(x)-f(y)\|+|\gamma(x)-\gamma(y)|\|f(x)-x\|+ 2|\gamma(y)|\|x-y\|\\
                                  & \leq\|x-y\|(\lip(f)+\lip(\gamma)(2\sigma+\|f(\eta)-\eta\|) +2\|\gamma\|_\infty)\\
                                  & \leq \|x-y\|(\lip(f)+\alpha + 2\sigma\alpha)\\ 
                                  & \leq \|x-y\| (\lip(f)+ 3 \alpha)
  \end{align*}
  shows that $\lip(\tilde{f}) \leq (\lip (f) + 3 \alpha)\leq 1 - \alpha <1$.

  Using the fact that the support of $\gamma$ is contained in $\overline{B}(\eta,\sigma)$, we get
  \begin{align*}
    \|\tilde{f}(x)-f(x)\| & = \gamma(x) \|f(x)-x\| \leq \gamma(x) (2\|x-\eta\|+\|f(\eta)-\eta\|)\\
                          & \leq \gamma(x) (2\sigma + \|f(\eta)-\eta\|) \leq \sigma\alpha <\varepsilon.
  \end{align*}
  Finally, note that by construction,
  \begin{equation}
    \tilde{f}(x) = \begin{cases} f(x) & x\not\in B(\eta,\sigma)\\ f(\eta) + \|\gamma\|_\infty(f(\eta)-\eta) & x=\eta.\end{cases}
  \end{equation}
\end{proof}

\begin{lemma}\label{lem:DisFixedPoints}
  Let $\varepsilon>0$, and let $f\colon D\to D$ and $g\colon D\to D$ be 
strict contractions with (unique) fixed points $\xi$ and $\eta$, respectively.
If $d_\infty(f, g) < \varepsilon$, then
  \[
    \|\xi-\eta\|<\min\left\{\frac{\varepsilon}{1-\lip f},\frac{\varepsilon}{1-\lip g}\right\}.
  \]
\end{lemma}

\begin{proof}
  Using the triangle inequality, we obtain
  \begin{align*}
    \|\xi-\eta\| & = \|f(\xi)-g(\eta)\|\leq \|f(\xi)-f(\eta)+f(\eta)-g(\eta)\|\\
                 & \leq \|f(\xi)-f(\eta)\|+\|f(\eta)-g(\eta)\| 
< (\lip f) \|\xi-\eta\| + \varepsilon,
  \end{align*}
  which is equivalent to
  \[
    \|\xi-\eta\| < \frac{\varepsilon}{1-\lip f}.
  \]
  This inequality, when combined with the analogous inequality for $g$, 
yields the claimed bound on $\|\xi-\eta\|$.
\end{proof}

\section{Pairs of nonexpansive mappings}

We begin this section with the following definition. 

\begin{definition}
  Given two nonexpansive mappings $f\colon D\to D$ and $g\colon D\to D$, we denote by $\{f,g\}$ the set-valued mapping defined by $x\mapsto\{f(x), g(x)\}$.
  We endow the space 
  \begin{equation}
    \mathcal{M} := \{\{f,g\}\colon f,g\colon D\to D\quad\text{nonexpansive}\}
  \end{equation}
  with the Hausdorff distance $H$ on the subsets of $\mathcal{M}^{1}$. 
By $\mathcal{N}$ we denote the subset of all pairs $\{f,g\}$ of 
strict contractions, that is, $\lip f,\lip g<1$. By $\mathcal{M}_{\infty}$ 
we denote the space $\mathcal{M}$ equipped with the metric
  \begin{equation}
    h_{\infty}(F,G) := \sup_{x\in D} h(F(x),G(x))
  \end{equation}
  of uniform convergence on $D$.
\end{definition}

\begin{remark}
  The metric space $(\mathcal{M}, H)$ is a complete metric space because 
it is a 
closed subset of the space of compact subsets of the complete metric space 
$\mathcal{M}^{1}$. Observe that given two elements  
$F,G\in\mathcal{M}$, $F=\{f_1,f_2\}$ and $G=\{g_1,g_2\}$, 
the Hausdorff distance $H$ satisfies
  \begin{equation}\label{eq:HAsMinMax}
    H(F,G)  = \min\{\max\{d_{\infty}(f_1,g_1),d_{\infty}(f_2,g_2)\}, 
\max\{d_{\infty}(f_1,g_2),d_{\infty}(f_2,g_1)\}\}
  \end{equation}
  ({\textit{cf.}}~Proposition~2.2 in~\cite[p.~1099]{Pia2015GenericProperties}). 
Moreover, it is easy to see that for each $x\in D$, the inequality 
$h(F(x),G(x))\leq H(F,G)$ is satisfied and hence
  \begin{equation}\label{eq:RelationBetweenMetrics}
    h_{\infty}(F,G) \leq H(F,G)
  \end{equation}
  for all $F,G\in \mathcal{M}$.
\end{remark}

\begin{proposition}
  The space $\mathcal{M}_{\infty} = (\mathcal{M}, h_{\infty})$ is complete.
\end{proposition}

\begin{proof}
  Let $\{F_n\}_{n\in\mathbb{N}}$ be a Cauchy sequence in 
$\mathcal{M}_{\infty}$. Since $\mathcal{M}_{\infty}$ is a topological 
subspace of the space of compact-valued nonexpansive mappings with the 
metric of uniform convergence, which is complete 
(see, for example, \cite{BMRZ2009GenericExistence}), this Cauchy sequence 
has a limit $F\colon D\to\mathcal{K}(D)$. As uniform convergence implies 
pointwise convergence and the space of sets with at most two elements is 
a closed subspace of the space of compact subsets of $D$, 
the point images of $F$ have at most two elements. It remains to be shown 
that there are two single-valued nonexpansive mappings $f$ and $g$ on $D$ 
such that $F=\{f,g\}$. In order to establish this assertion, We use the 
following iterative argument.
  We start by setting
  \[
    A_{0} := \{x\in D\colon |F(x)| = 1\} = \{x\in D\colon 
F(x) = \{F_1(x)\}\},
  \]
  
  \[
    f_{0}\colon A_0 \to D, \quad x\mapsto F_1(x)\qquad \text{and}\qquad g_0 := f_0.
  \]
  Given $\varepsilon>0$, we set
  \[
    A_{\varepsilon} := \{x\in D\colon \diam F(x) > \varepsilon\}.
  \]
  In other words, by $A_\varepsilon$ we denote the set of points for which 
the two elements of the point image of $F$ are at least $\varepsilon$ 
apart. Now pick $\varepsilon_1>0$ small enough and choose an $n_1\in\mathbb{N}$ such that $h_{\infty}(F_{n_1},F) < \frac{\varepsilon_1}{3}$. Let $h_1^{1},h_2^{1}\colon D\to D$ be nonexpansive mappings such that 
$F_{n_1}=\{h_1^{1},h_2^{1}\}$. In the sequel, we use the notation 
$F(x) = \{F_1(x), F_2(x)\}$. We now define
  \[
    f_1\colon A_{\varepsilon_1}\to D, \qquad x\mapsto
    \begin{cases}
      F_1(x)&: \|F_1(x)-h_1^{1}(x)\|<\frac{\varepsilon_1}{3}\\
      F_2(x)&: \|F_2(x)-h_1^{1}(x)\|<\frac{\varepsilon_1}{3}\\      
    \end{cases}
  \]
  and
  \[
    g_1\colon A_{\varepsilon_1}\to D, \qquad x\mapsto
    \begin{cases}
      F_1(x)&: \|F_1(x)-h_2^{1}(x)\|<\frac{\varepsilon_1}{3}\\
      F_2(x)&: \|F_2(x)-h_2^{1}(x)\|<\frac{\varepsilon_1}{3}\\      
    \end{cases}.
  \]
  Note that the condition $\|F_1(x)-F_2(x)\|>\varepsilon_1$ ensures 
that both $f_1$ and $g_1$ are well defined and that the 
inequality $h_{\infty}(F_n,F) < \frac{\varepsilon_1}{3}$ ensures that 
for each $x\in A_{\varepsilon_1}$, at least one element of $F(x)$ is close 
enough to $h_1(x)$ and $h_2(x)$, respectively.

  Now fix $m\in\mathbb{N}$, and assume that $f_m$ and $g_m$ are already 
defined, and satisfy
  \[
    F(x) = \{f_m(x),g_m(x)\}
  \]
  for each $x\in A_{\varepsilon_m}$. Moreover, we assume that we have picked $n_{m}\in\mathbb{N}$ and nonexpansive mappings $h_1^{m},h_2^{m}\colon D\to D$ such that $F_{n_m}=\{h_1^{m},h_2^{m}\}$,
  \[
    \|h_1^{m}(x)-f_m(x)\|\leq \frac{\varepsilon_{m}}{3} \qquad\text{and}\qquad
    \|h_2^{m}(x)-g_{m}(x)\|\leq \frac{\varepsilon_{m}}{3}
  \]
  for each $x\in A_{\varepsilon_{m-1}}$. We choose an $\varepsilon_{m+1}\in (0,\frac{\varepsilon_{m}}{2})$ and pick a natural number $n_{m+1}> n_{m}$ such that~$h_{\infty}(F_{n_{m+1}},F) < \frac{\varepsilon_{m+1}}{3}$. Now we can find nonexpansive mappings \[
    h_1^{m+1}, h_{2}^{m+1}\colon D\to D
  \]
  such that $F_{n_{m+1}}=\{h_1^{m+1},h_2^{m+1}\}$,
  \[
    \|h_{1}^{m}(x)-h_1^{m+1}(x)\| \leq 
\frac{1}{3} (\varepsilon_m+\varepsilon_{m+1})\quad \text{and} \quad 
\|h_{2}^{m}(x)-h_2^{m+1}(x)\| \leq 
\frac{1}{3} (\varepsilon_m+\varepsilon_{m+1}).
  \]
  When combined with the inverse triangle inequality, these inequalities 
also yield
  \begin{equation}\label{eq:ComplInConst1}
    \|h_{1}^{m}(x)-h_2^{m+1}(x)\| \geq \varepsilon_m -\frac{1}{3} (\varepsilon_m+\varepsilon_{m+1}) > \frac{1}{3} (\varepsilon_m+\varepsilon_{m+1})
  \end{equation}
  and
  \begin{equation}\label{eq:ComplInConst2}
    \|h_{2}^{m}(x)-h_1^{m+1}(x)\| \geq \varepsilon_m -\frac{1}{3} (\varepsilon_m+\varepsilon_{m+1})  > \frac{1}{3} (\varepsilon_1+\varepsilon_{m+1})
  \end{equation}
  as $\varepsilon_m > 2\varepsilon_{m+1}$. Next we define
  \[
    f_{m+1}\colon A_{\varepsilon_{m+1}}\to D, \qquad x\mapsto
    \begin{cases}
      F_1(x)&: \|F_1(x)-h_1^{m+1}(x)\|<\frac{\varepsilon_{m+1}}{3}\\
      F_2(x)&: \|F_2(x)-h_1^{m+1}(x)\|<\frac{\varepsilon_{m+1}}{3}\\      
    \end{cases}
  \]
  and
  \[
    g_{m+1}\colon A_{\varepsilon_{m+1}}\to D, \qquad x\mapsto
    \begin{cases}
      F_1(x)&: \|F_1(x)-h_2^{m+1}(x)\|<\frac{\varepsilon_{m+1}}{3}\\
      F_2(x)&: \|F_2(x)-h_2^{m+1}(x)\|<\frac{\varepsilon_{m+1}}{3}\\      
    \end{cases}.
  \]
  Again the conditions on $A_{\varepsilon_{m+1}}$ and $F_{n_{m+1}}$ ensure that these mappings are well defined. Note that 
inequalities~\eqref{eq:ComplInConst1} and~\eqref{eq:ComplInConst2} 
imply that
  \[
    f_{m+1}(x) = f_{m}(x)\qquad\text{and}\qquad g_{m+1}(x) = g_{m}(x)
  \]
  for all $x\in A_{\varepsilon_m}$.

  The result of this inductive construction is a sequence 
of sets $\{A_{\varepsilon_m}\}$ such that 
  \[
    A_{\varepsilon_{m}} \subset A_{\varepsilon_{m+1}} \qquad\text{and}\qquad
    D = \bigcup_{m=0}^{\infty} A_{\varepsilon_m}
  \]
  and a sequence $\{f_{m},g_{m}\}$. We now set $k(x) := \min\{m\in\mathbb{N}\colon x\in A_{\varepsilon_{m}}\}$ and define two mappings $f\colon D\to D$ and $g\colon D\to D$ by
  \[
    f(x) := f_{k(x)}(x)\qquad \text{and}\qquad g(x) := g_{k(x)}(x).
  \]
  We finish the proof by showing that $f$ is nonexpansive. The argument 
for $g$ is completely similar. Let $x,y\in D$ be given. Assume that 
neither one of the points are in the set $A_0$. For every $\varepsilon>0$, 
there is an $\varepsilon_{m}<\varepsilon$ such that $x,y \in A_{\varepsilon_{m}}$ and hence
  \begin{align*}
    \|f(x)-f(y)\| &\leq \|h_{1}^{m+1}(x)-h_{1}^{m+1}(y)\| + \|f(x)-h_{1}^{m+1}(x)\| + \|f(y)-h_{1}^{m+1}(y)\|\\
                  &\leq \|x-y\| + \varepsilon.
  \end{align*}
  If one of the points, say $x$, is contained in $A_0$, then the 
inequality $h(F_{n_m}(x),F(x))<\frac{\varepsilon_m}{3}$ implies that
  \[
    \|h_1^{m}(x)-f(x)\| \leq \frac{\varepsilon_m}{3}\qquad \text{and}\qquad
    \|h_2^{m}(x)-f(x)\| \leq \frac{\varepsilon_m}{3}
  \]
  because $F(x)=\{f(x)\}$. Hence in any case, we end up with
  \[
    \|f(x)-f(y)\|\leq \|x-y\| + \varepsilon
  \]
  for every $\varepsilon>0$. Letting $\varepsilon \to 0^{+}$, we arrive at
the claimed result.
\end{proof}

\begin{lemma}\label{lem:SetValDiff}
  Let $\varepsilon>0$, and let $f_1$, $f_2$, $g_1$ and $g_2$ be nonexpansive 
self-mappings of $D$ with $d_\infty(f_1, f_2), \; d_\infty(g_1, g_2) < \varepsilon$. 
Then the set-valued mappings $F_1$ and $F_2$ defined by
  \begin{equation}
    F_1(x) := \{f_1(x),g_1(x)\}\qquad\text{and}\qquad F_2(x) := \{f_2(x),g_2(x)\}, \qquad x\in D,
  \end{equation}
  satisfy $h(F_1(x),F_2(x))<\varepsilon$ and $H(F_1,F_2)<\varepsilon$.
\end{lemma}

\begin{proof}
  This is an easy consequence of~\eqref{eq:HAsMinMax} 
and of~\eqref{eq:RelationBetweenMetrics}; {\textit{cf.}}~Remark~2.3 
in~\cite{Pia2015GenericProperties}.
\end{proof}

\begin{lemma}\label{lem:SetValDiffInv}
  Let $\varepsilon>0$, let $f_1$, $f_2$, $g_1$ and $g_2$ be nonexpansive 
self-mappings of $D$, and let  $F_1$ and $F_2$ be the set-valued mappings 
defined by
  \begin{equation}
    F_1(x) := \{f_1(x),g_1(x)\}\qquad\text{and}\qquad F_2(x) := \{f_2(x),g_2(x)\}, \qquad x\in D.
  \end{equation}
  Then the inequality $H(F_1,F_2)<\varepsilon$ implies that
  \[
    d_\infty(f_1, f_2), d_\infty(g_1, g_2) < 
\varepsilon\quad \text{or} \quad  
d_\infty(f_1, g_2), d_\infty(g_1, f_2) < \varepsilon.
  \]
\end{lemma}

\begin{proof}
  See Proposition~2.2 in~\cite{Pia2015GenericProperties}.
\end{proof}

\begin{remark}
  Note that the above lemma fails if we replace the Hausdorff distance $H$ by the metric of uniform convergence. For example, consider $D=[-1,1]^{3}$, which is a bounded, closed and convex subset of $\mathbb{R}^{3}$, and the mappings
  \begin{align*}
    f_1\colon D\to D &\qquad (x,y,z)\mapsto (x,0,\varepsilon/2),\\
    g_1\colon D\to D &\qquad (x,y,z)\mapsto (-x,0,-\varepsilon/2),\\
    f_2\colon D\to D &\qquad (x,y,z)\mapsto \begin{cases}(x,0,0)&x\leq 0\\(-x,0,0) &x>0\end{cases}\\
    g_2\colon D\to D &\qquad (x,y,z)\mapsto \begin{cases}(-x,0,0)&x\leq 0\\(x,0,0)&x>0\end{cases}.    
  \end{align*}
  For these mappings, we obtain
  \[
    h(F_1(\xi),F_2(\xi)) = h(\{f_1(\xi),g_1(\xi)\},\{f_2(\xi),g_2(\xi)\}) < \varepsilon,
  \]
  but the inequalities 
  \[
    \|f_1(\xi)-f_2(\xi)\| \geq 2 \quad \text{and} \quad 
\|g_1(\xi)-f_2(\xi)\| \geq 2
  \]
  show, for $\varepsilon$ small enough, that a selection in the 
spirit of the above lemma is not possible.
\end{remark}

Given two nonexpansive mapping $f,g\colon D\to D$, we consider the mapping
\begin{equation}
  F\colon D \to \binom{D}{\leq 2}, \quad x\mapsto \{f(x),g(x)\}.
\end{equation}
Then $F\in\mathcal{M}$ by the definition of $\mathcal{M}$.

\begin{definition} Let $F\in\mathcal{M}$. A sequence $\{x_n\}_{n\in\mathbb{N}}$, where 
  \begin{equation}\label{eq:DefSequence}
    x_{n+1}\in P_{F(x_n)}(x_n)
  \end{equation}
  for $n\in\mathbb{N}$, is called a \emph{sequence of successive approximations with respect to $F$}. The sequence $\{x_n\}_{n\in\mathbb{N}}$ is called \emph{regular} if $P_{F(x_n)}(x_n)$ is a singleton for all $n\in\mathbb{N}$.
\end{definition}

\begin{lemma}
  The set $\mathcal{N}\subset\mathcal{M}$ is dense in $\mathcal{M}$.
\end{lemma}

\begin{proof}
  This result follows from the corresponding result for single-valued mappings; see, for example,~\cite{DM1976Convergence} and Lemma~\ref{lem:SetValDiff}.
\end{proof}

\begin{proposition}\label{prop:Convergence}
  Let $F\colon D\to \mathcal{K}(D)$ be a Lipschitz mapping with Lipschitz constant $L<1$. Then every sequence $\{x_n\}_{n\in\mathbb{N}}$ with $x_{n+1}\in P_{F(x_n)}(x_n)$ converges to a fixed point of $F$.
\end{proposition}

\begin{proof}
  For $n\in\mathbb{N}$, we have
  \begin{equation}\label{eq:FLipIterate}
    \|x_{n+1}-x_{n}\| = d(x_n, F(x_n)) \leq h(F(x_{n-1}), F(x_n)) \leq L \|x_n-x_{n-1}\|
  \end{equation}
  because $x_n\in P_{F(x_{n-1})}(x_{n-1})$ and $F$ is $L$-Lipschitz. Hence 
for $k\in\mathbb{N}$, we obtain
  \begin{equation}\label{eq:BanachItBound}
    \begin{aligned}
      \|x_{n+k}-x_{n}\| & \leq \sum_{j=1}^{k} \|x_{n+j}-x_{n+j-1}\| \leq  \|x_{n}-x_{n-1}\| \sum_{j=1}^{k} L^j \leq \|x_1-x_0\| \sum_{j=1}^{k} L^{j+n-1}\\ 
      & \leq \|x_1-x_0\| \sum_{j=n}^{\infty} L^j  = \|x_1-x_0\| \frac{L^{n}}{1-L},
    \end{aligned}
  \end{equation}
  that is, $\{x_n\}_{n\in\mathbb{N}}$ is a Cauchy sequence. Since $X$ is complete, this sequence has a limit $x^*$. It remains to be shown that $x^*$ is a fixed point of $F$. Indeed, for each $\varepsilon>0$, there is a point $x_n$ such that $\|x_n-x^*\|<\varepsilon/3$ and $\|x_n-x_{n+1}\|<\varepsilon/3$. We have
  \begin{align*}
    d(x^*, F(x^*)) &\leq \|x_n-x^*\| + d(x_n,F(x_n)) + h(F(x_n),F(x^*)) \\
                   & \leq \|x_n-x^*\| + \|x_n-x_{n+1}\| + L \|x_n-x^*\| < \varepsilon.
  \end{align*}
  As this inequality is true for any $\varepsilon>0$, it follows that $d(x^*, F(x^*)=0$ and therefore $x^*\in F(x^*)$ because $F(x^*)$ is closed.
\end{proof}

\begin{proposition}\label{prop:NotRepeat}
  Let $F\colon D\to \mathcal{K}(D)$ be a Lipschitz mapping with Lipschitz constant $L<1$ and let $\{x_n\}_{n\in\mathbb{N}}$ be as in~\eqref{eq:DefSequence}. If there are $k,p\in\mathbb{N}$, $p\geq 1$, with $x_{k+p}=x_{k}$, then $x_{k+j}=x_{k}\in\fix F$ for all $j\in\mathbb{N}$.
\end{proposition}

\begin{proof}
  Since $x_k=x_{k+p}$, we have
  \[
    \|x_{k+1}-x_k\| = d(x_k,F(x_k)) = d(x_{k+p},F(x_{k+p})) = \|x_{k+p+1}-x_{k+p}\|
  \]
  and
  \[
    \|x_{k+p+1}-x_{k+p}\| \leq L^{p-1} \|x_{k+1}-x_{k}\|
  \]
  by~\eqref{eq:FLipIterate}. Combining these inequalities, we obtain
  \[
    \|x_{k+1}-x_k\|\leq L^{p-1} \|x_{k+1}-x_{k}\|,
  \]
  which implies that $x_{p+1}=x_{p}$ and $x_p\in\fix F$ since $x_p\in F(x_p)$.
\end{proof}

\begin{proposition}\label{prop:OnlyFinitelyManyBad}
  Let $f$ and $g$ be strict contractions on $D$ with distinct fixed points $\xi$ and~$\eta$, respectively. Let $\{x_n\}_{n\in\mathbb{N}}$ be a sequence as in~\eqref{eq:DefSequence}. Then the number of elements of $\{x_n\}_{n\in\mathbb{N}}$ for which
  \begin{equation}
    \|x_n-f(x_n)\|=\|x_x-g(x_n)\|
  \end{equation}
  is finite. Moreover, there is $z\in\{\xi,\eta\}$ and an $r_0$ so that for every $0<r<r_0$, there is an index $N\in\mathbb{N}$ such that $x_n\in B[z,r]$ for all $n\geq N$ and $x_n\not\in B[z,r]$ for $n=0,\ldots,N-1$.
\end{proposition}

\begin{proof}
  By Proposition~\ref{prop:Convergence}, we know that the sequence $\{x_n\}_{n\in\mathbb{N}}$ converges to a fixed point of $F$. Without loss of generality, we may assume that the sequence converges to~$\xi$.
  The existence of a $k\in\mathbb{N}$, where $x_k$ is a fixed point of $F$, implies that the rest of the sequence remains constant and, since the fixed points of $f$ and $g$ are distinct, we may set~$N:=k$ in order to satisfy the claimed assertion.

  Therefore it remains to consider the case where $x_n\not\in\{\xi,\eta\}$ for all $n\in\mathbb{N}$. We set $\alpha := \min\{\|f(\xi)-g(\xi)\|, \|x_0-\xi\|\}$ and observe that $\alpha >0$ since the fixed points of $f$ and $g$ are distinct, and $x_0\neq\xi$. The continuity of $f$ and $g$ at $\xi$ implies the existence of an $r_0>0$ with $0<r_0<\frac{\alpha}{4}$ such that $\|x-\xi\|<r_0$ implies that $\|f(x)-f(\xi)\|<\frac{\alpha}{4}$ and $\|g(x)-g(\xi)\|<\frac{\alpha}{4}$.

  For $x\in B(\xi,r_0)$, the triangle inequality, when combined with $f(\xi)=\xi$, implies that
  \begin{equation}\label{eq:UniqueByDistance}
    \begin{aligned}
      \|f(x)-x\| & \leq \|f(x)-f(\xi)\|+\|\xi-x\| \leq r_0 + \frac{\alpha}{4} <\frac{\alpha}{2}\qquad\text{and}\\
      \|g(x)-x\| & \geq \|g(\xi)-f(\xi)\| - \|g(\xi)-g(x)\| - \|\xi-x\| \geq \alpha - \frac{\alpha}{4} - r_0 \geq \frac{\alpha}{2},
    \end{aligned}
  \end{equation}
  and hence $\|f(x)-x)\|<\|g(x)-x\|$. Given $0<r<r_0$, we set 
  \[
    N:=\min\{n\in\mathbb{N}\colon \|x_n-\xi\| \leq r\},
  \]
  which exists because $x_n\to \xi$. This implies $x_N\in B(\xi,r)$ and 
$x_{N+1}=f(x_N)$. Since $f$ is nonexpansive, we get 
$\|x_{N+1}-\xi\|=\|f(x_{N})-f(\xi)\|\leq \|x_{N}-\xi\| \leq r$. Therefore 
we can deduce inductively that $x_n\in B(\xi,r)$ and $x_{n+1}=f(x_n)$ for 
all~$n\geq N$. Finally, we may use this bound, the inequality $r<r_0$ 
and~\eqref{eq:UniqueByDistance} to obtain that the set of elements of $\{x_n\}_{n\in\mathbb{N}}$ for which
  \[
    \|x_n-f(x_n)\|=\|x_n-g(x_n)\|
  \]
  is contained in $\{x_0,x_1,\ldots,x_{N-1}\}$ and therefore is finite, as 
asserted. 
\end{proof}

\section{Generic Convergence}
\begin{proposition}\label{Prop:CloseRegular}
  Let $f$ and $g$ be strict contractions on $D$ with distinct fixed points $\xi$ and~$\eta$, respectively. Moreover, assume that there is $\theta_0>1$ such that 
  \begin{equation}
    \theta f(x) + (1-\theta) x\in D \qquad \text{and} \qquad \theta g(x) + (1-\theta) x\in D
  \end{equation}
  for every $x\in D$ and every $\theta<\theta_0$. Let $\{x_n\}_{n\in\mathbb{N}}$ be a sequence as in~\eqref{eq:DefSequence}. Then for each $\varepsilon>0$, there are strict contractions $\varphi$ and $\psi$ on $D$ such that $H(\{f,g\},\{\varphi,\psi\})<\varepsilon$, all metric projections $P_{\{\varphi(x_n),\psi(x_n)\}}(x_n)$ are unique and $\{x_n\}_{n\in\mathbb{N}}$ satisfies~\eqref{eq:DefSequence} for $\{\varphi,\psi\}$.
\end{proposition}

\begin{proof}
  Without loss of generality, we may assume that $x_n\to\xi$. By Proposition~\ref{prop:OnlyFinitelyManyBad}, there is an $r_0>0$ such that for all $0<r<r_0$, there is an $N\in\mathbb{N}$ so that $x_n\in B[\xi,r]$ for $n\geq N$ and the set of elements of $\{x_n\}_{n\in\mathbb{N}}$ with
  \[
    \|x_n-f(x_n)\|=\|x_n-g(x_n)\|
  \]
  is contained in $\{x_0,x_1,\ldots,x_{N-1}\}$. Since $x_0,x_1,\ldots,x_{N-1}\not\in B[\xi,r]$, Proposition~\ref{prop:NotRepeat} implies that none of these points coincide. Define
  \[
    \sigma := \min\left\{1, \frac{\|x_i-x_j\|}{2}, \|x_i-\xi\|-r\colon i,j=0,1,\ldots N-1\right\}
  \]
  and note that the above arguments imply that $\sigma >0$. By the definition of $\sigma$, we see that the balls $B(x_i,\sigma)$, $\sigma=0,1,\ldots,N-1$, and $B(\xi,r)$ are pairwise disjoint.

  Next, we define the mappings $\varphi$ and $\psi$ inductively: we start with setting $\varphi_0:=f$ and $\psi_0:=g$. Now for a fixed  $k\in\{1,\ldots,N-1\}$, assume that $\varphi_{k-1}$ and $\psi_{k-1}$ have already been defined. 

  If $\|f(x_{k-1})-x_{k-1}\|\neq\|g(x_{k-1})-x_{k-1}\|$, then we set $\varphi_k:=\varphi_{k-1}$ and $\psi_k:=\psi_{k-1}$. Otherwise, that is, if $\|f(x_{k-1})-x_{k-1}\|=\|g(x_{k-1})-x_{k-1}\|$, then we proceed as follows: 

  If $x_{k}=f(x_{k-1})$, then we set $\varphi_{k}:=\varphi_{k-1}$ and use Lemma~\ref{lem:perturbation} to obtain a strict contraction $\psi_k$ with the following properties: 
  \begin{enumerate}[(i)]
    \item \label{eq:PerCond1} $\|\psi_k(x)-\psi_{k-1}(x)\| < \varepsilon$ 
    for all $x\in B(x_{k-1},\sigma)$,
    \item \label{eq:PerCond2} $\psi_k(x)=\psi_{k-1}(x)$ 
    for all $x\not\in B(x_{k-1},\sigma)$ and 
    \item \label{eq:PerCond3} $\psi_k(x_{k-1}) = g(x_{k-1}) + c(g(x_{k-1})-x_{k-1})$ for some $c>0$ and hence 
      \[
        \|\varphi_{k}(x_{k-1})-x_{k-1}\| = \|f(x_{k-1})-x_{k-1}\| = \|g(x_{k-1})-x_{k-1}\| < \|\psi_{k}(x_{k-1})-x_{k-1}\|.
      \]
  \end{enumerate}
  Since the balls $B(x_i,\sigma)$, $\sigma=0,1,\ldots,N-1$, and $B(\xi,r)$ 
are pairwise disjoint, \eqref{eq:PerCond1}~implies that 
$\|\psi_{k}(x)-g(x)\|<\varepsilon$ for all $x\in D$, 
and \eqref{eq:PerCond2}~implies that $\psi_k(x_n)=g(x_n)$ for all $n\geq N$ 
and $\psi_{k}(x_n)=\psi_n(x_n)$ for all $n=0,1,\ldots,k-1$. 
Finally, note that~\eqref{eq:PerCond3} implies that 
  \[
    P_{\{\varphi_k(x_{k-1}),\psi_{k}(x_{k-1})\}}(x_{k-1})= \{x_k\},
  \]
  that is, the metric projection is unique. 

  If on the other hand,  $x_{k}=g(x_{k-1})$, then we set $\psi_{k}:=\psi_{k-1}$ and use the above procedure to obtain a strict contraction $\varphi_k$ which satisfies $\varphi_k(x_n)=f(x_n)$ for all $n\geq N$, $\varphi_k(x_n)=\varphi_n(x_n)$ for the indices~$n=0,1,\ldots,k-1$, $\|\varphi_k(x_{k-1})-x_{k-1}\|> \|\psi_k(x_{k-1})-x_{k-1}\|$ and $\|f(x)-\varphi_k(x)\|<\varepsilon$ for all $x\in D$.

  Setting $\varphi:=\varphi_N$, $\psi:=\psi_N$ and using 
Lemma~\ref{lem:SetValDiff}, we finish the proof.
\end{proof}

\begin{proposition}\label{Prop:Holes}
    Let $f$ and $g$ be strict contractions on $D$ with distinct fixed points $\xi$ and~$\eta$, respectively. In addition, let $\{x_n\}_{n\in\mathbb{N}}$ be a regular sequence of successive approximations for $\{f,g\}$. Then there are an $\varepsilon_0>0$ and an $\alpha>0$ so that for all $0<\varepsilon<\varepsilon_0$ and all $\{\varphi,\psi\}\in B(\{f,g\},\alpha\varepsilon)$, every sequence $\{y_n\}_{n\in\mathbb{N}}$ of successive approximations for $\{\varphi,\psi\}$ with $y_0=x_0$ is regular and satisfies $\|x_n-y_n\|\leq \varepsilon$ for all $n\in\mathbb{N}$.
\end{proposition}

\begin{proof}
  By assumption, the sequence $\{x_n\}_{n\in\mathbb{N}}$ converges. Without loss of generality, we may assume that it converges to a fixed point~$\xi$ of~$f$. Again by assumption, $\xi$ is not a fixed point of~$g$. We set 
  \[
    \varepsilon_0 := \min\left\{\frac{\|g(\xi)-\xi\|}{3},\frac{1}{2}, d_\infty(f, g) \right\}
  \]
  and assume $0<\varepsilon<\varepsilon_0$ to be given. Since $x_n\to \xi$, 
there is an $N\in\mathbb{N}$ such that $x_n\in B(\xi,\varepsilon/4)$ for 
all $n\geq N$. We set
  \[
    \sigma := \min\left\{1,\big|\|f(x_k)-x_k\|-\|g(x_k)-x_k\|\big|\colon k=0,\ldots, N\right\}
  \]
  and
  \[
    \alpha:=\min\left\{\frac{1-\max\{\lip f,\lip g\}}{2}, \frac{\sigma}{4N}\right\},
  \]
  which is positive since $f$ and $g$ are strict contractions and $\{x_n\}_{n\in\mathbb{N}}$ is a regular sequence.

  Now let $\varphi\in B(f,\alpha\varepsilon)$ and $\psi\in B(g,\alpha\varepsilon)$ be arbitrary. Observe that for $z\in B(\xi,\varepsilon/2)$, 
the conditions $d_\infty(\varphi,f) < \alpha\varepsilon$ and 
$\|z-\xi\|<\varepsilon/2$, when combined with the triangle inequality, imply that
  \begin{equation}\label{eq:RestOfSequence}
    \|\varphi(z)-\xi\| < \alpha\varepsilon + \|f(z)-f(\xi)\| \leq (2\alpha+\lip f) \frac{\varepsilon}{2} \leq \frac{\varepsilon}{2},
  \end{equation}
  that is, $\varphi$ maps $B(\xi,\varepsilon/2)$ into itself.

  Let the sequence $\{y_k\}_{k\in\mathbb{N}}$ satisfy
  \[
    y_0:=x_0\qquad\text{and}\qquad
    y_{k+1}\in P_{\{\varphi(y_k),\psi(y_k)\}}(y_k).
  \]

  We show by induction that $\|x_k-y_k\|< k\alpha\varepsilon$ for 
$k=0,\ldots,N$. For $k=0$, this statement is true by the definition of 
$y_0$. 
Assume now that we have already proved the bound for the difference of $x_k$ and $y_k$.
  If $x_{k+1}=f(x_k)$, we get
  \begin{equation}
    \begin{aligned}
      \|\varphi(y_k)-x_{k+1}\| & \leq \|\varphi(y_k)-f(y_k)\|+\|f(y_k)-f(x_k)\|\\
                              & < \alpha\varepsilon + k\alpha\varepsilon = (k+1)\alpha\varepsilon
  \end{aligned}
  \end{equation}
  and 
  \begin{equation}\label{eq:NewSeqReg}
  \begin{aligned}
    \|\varphi(y_{k})-y_{k}\| & \leq \|f(x_k)-x_k\| + \|f(x_k)-\varphi(x_k)\| + \|\varphi(x_k)-\varphi(y_k)\| + \|x_k-y_k\|\\
                             & \leq \|f(x_k)-x_k\| + \alpha\varepsilon + 2\|x_k-y_k\|\leq  \|f(x_k)-x_k\| + (2k+1)\alpha\varepsilon\\
                             & \leq \|g(x_k)-x_k\|+(2k+1)\alpha\varepsilon-4N\alpha \leq \|\psi(y_k)-y_k\|+(2k+1-4N)\alpha\\
                             &  < \|\psi(y_k)-y_k\|
  \end{aligned}
  \end{equation}
  because $\|g(x_k)-x_k\|-\|f(x_k)-x_k\| \geq 4N\alpha $ and $\varepsilon\leq \frac{1}{2}$. Hence 
  \[
    y_{k+1}=\varphi(y_k)\quad\text{and}\quad \|y_{k+1}-x_{k+1}\|\leq (k+1)\alpha\varepsilon, 
  \]
  as claimed. A similar argument works for the case where $x_{k+1}=g(x_k)$. 

  Now we may use $\alpha\leq\frac{1}{4N}$ to deduce that
  \begin{equation}\label{eq:DistBeginning}
    \|y_k-x_k\|\leq \frac{k}{4N}\varepsilon\leq\varepsilon/4
  \end{equation}
  for $k=0,1,\ldots,N$. 

  For $z\in B(\xi,\varepsilon/2)$, the bound
  \[
    \|\psi(z)-z\| > \|g(\xi)-\xi\| - \alpha\varepsilon - 2\|\xi-z\| > \|g(\xi)-\xi\|-2\varepsilon > \varepsilon
  \]
  shows, when combined with~\eqref{eq:RestOfSequence}, that $\|\psi(z)-z\|>\|\varphi(z)-z\|$ in this case. Since $y_N\in B(\xi,\varepsilon/2)$ by~\eqref{eq:DistBeginning} and the definition of $N$, this implies that $y_{k+1}=\varphi(y_k)$ for $k\geq N$ and that the whole sequence $\{y_n\}_{n\in\mathbb{N}}$ is regular. Again by~\eqref{eq:DistBeginning} and since $\varphi$ maps $B(\xi,\varepsilon/2)$ into itself, we conclude that $y_k\in B(\xi,\varepsilon/2)$ for all $k\geq N$ and hence $\|x_k-y_k\|\leq \diam B(\xi,\varepsilon/2)=\varepsilon$, as asserted.
\end{proof}

\begin{theorem}\label{thm:GenericConv1}
  Let $u\in D$. There is a residual set $\mathcal{M}_*\subset\mathcal{M}$ such that for every mapping $\{\varphi,\psi\}\in\mathcal{M}_*$, the sequence of successive approximations is regular (and therefore unique) and converges to a fixed point of $\varphi$ or of $\psi$.
\end{theorem}

\begin{proof}
  We denote by $\mathcal{N}_{*}$ the set of all mappings $\{f,g\}\in\mathcal{N}$ such that $f$ and $g$ have distinct fixed points and every sequence $\{x_n\}_{n\in\mathbb{N}}$ of successive approximations for $\{f,g\}$ with $x_0=u$ is regular.

  Lemmata~\ref{lem:DiffFixedPoint},~\ref{lem:PropD} and~\ref{lem:SetValDiff}, when combined with Proposition~\ref{Prop:CloseRegular}, imply that $\mathcal{N}_{*}$ is dense in $\mathcal{N}$. Since $\mathcal{N}$ is dense in $\mathcal{M}$, we may deduce that $\mathcal{N}_{*}$ is a dense subset of $\mathcal{M}$. We define
  \begin{equation}\label{eq:GoodSet}
    \mathcal{M}_{*} := \bigcap_{i=1}^{\infty} \;\;\bigcup_{\{f,g\}\in\mathcal{N}_*} B_{\mathcal{M}}\left(\{f,g\},\min\left\{\frac{\alpha_{\{f,g\}}\varepsilon_{0,\{f,g\}}}{2},\frac{\alpha_{\{f,g\}}}{i}\right\}\right),
  \end{equation}
  where $\alpha_{\{f,g\}}$ and $\varepsilon_{0,\{f,g\}}$ are given by Proposition~\ref{Prop:Holes}. Since $\mathcal{N}_*\subset\mathcal{M}$ is dense, we see immediately that $\mathcal{M}_*$ is a dense $G_\delta$-set.

  Proposition~\ref{Prop:Holes} guarantees that every sequence of 
successive approximations $\{x_n\}_{n\in\mathbb{N}}$ with respect to any $\{\varphi,\psi\}\in\mathcal{M}_*$ and where $x_0=u$, is regular. 
Therefore it remains to be shown that $\{x_n\}_{n\in\mathbb{N}}$ converges 
to a fixed point of $\varphi$ or of $\psi$.

  To this end, we first show that $\{x_n\}_{n\in\mathbb{N}}$ is a Cauchy 
sequence. 
Given $\varepsilon>0$, we choose a natural number $i > \frac{3}{\varepsilon}$. For all $\{\varphi,\psi\}\in\mathcal{M}_{*}$, by definition, there is an element $\{f,g\}\in\mathcal{N}_{*}$ with 
  \[  H(\{\varphi,\psi\},\{f,g\})<\min\left\{\frac{\alpha_{\{f,g\}}\varepsilon_{0,\{f,g\}}}{2},\frac{\alpha_{\{f,g\}}}{i}\right\}.
  \]
  Hence, by Proposition~\ref{Prop:Holes}, the sequence $\{y_n\}_{n\in\mathbb{N}}$ of successive approximations with respect to $\{f,g\}$ with initial point $y_0=u$ satisfies
  \begin{equation}
    \|x_n-y_n\| \leq \min\left\{\frac{\varepsilon_{0,\{f,g\}}}{2},\frac{1}{i}\right\}\leq \frac{1}{i}
  \end{equation}
  for all $n\in\mathbb{N}$.
  Since $\{y_n\}$ is a convergent sequence, there is an $N\in\mathbb{N}$ 
such 
that $\|y_n-y_m\|<\frac{\varepsilon}{3}$ for all $m,n\geq N$. 
Therefore, by the triangle inequality, we have
  \begin{align*}
    \|x_n-x_m\| & \leq \|x_n-y_n\|+\|y_n-y_m\|+\|y_m-x_m\| \leq \frac{2}{i} +\frac{\varepsilon}{3} <\varepsilon
  \end{align*}
  for all $m,n\geq N$, that is, $\{x_n\}_{n\in\mathbb{N}}$ is indeed a 
Cauchy sequence.

  Since $D\subset X$ is closed, the sequence $\{x_n\}_{n\in\mathbb{N}}$ converges to a point $x^*\in D$. Hence for all $\varepsilon>0$, there is an $N\in\mathbb{N}$ such that $\|x_n-x^*\|<\frac{\varepsilon}{2}$ for all $n\geq N$. We can use this bound to obtain
  \begin{align*}
    d(x^*,\{\varphi(x^*),\psi(x^*)\}) &\leq d(x^*, \{\varphi(x_n),\psi(x_n)\}) + h(\{\varphi(x_n),\psi(x_n)\}, \{\varphi(x^*),\psi(x^*)\})\\ 
                                      &\leq \|x_{n+1}-x^*\|+\|x_n-x^*\| < \varepsilon
  \end{align*}
  because $\{\varphi,\psi\}\in\mathcal{M}$ and $x_{n+1}\in \{\varphi(x_n),\psi(x_n)\}$. Since $\varepsilon>0$ is arbitrary, we may conclude that $d(x^*,\{\varphi(x^*),\psi(x^*)\})=0$ and therefore either $\varphi(x^*)=x^*$ or $\psi(x^*)=x^*$. In other words, the sequence
  $\{x_n\}_{n\in\mathbb{N}}$ converges to a fixed point of $\varphi$ or $\psi$.
\end{proof}

The above result is also true for the metric of uniform convergence, or in 
other words, for the space $\mathcal{M}_{\infty} = 
(\mathcal{M}, h_{\infty})$. However, in order to prove this statement, we 
need some preparation. Since $h_{\infty}(F,G)\leq H(F,G)$, the inequality  
$H(F,G)<\varepsilon$ implies that $h_{\infty}(F,G)<\varepsilon$, and so we 
only need the following variant of Proposition~\ref{Prop:Holes} for this 
new case.

\begin{proposition}\label{Prop:HolesUniform}
  Let $f$ and $g$ be strict contractions on $D$ with distinct fixed points $\xi$ and~$\eta$, respectively. In addition, let $\{x_n\}_{n\in\mathbb{N}}$ be a regular sequence of successive approximations for $\{f,g\}$. Then there are an $\varepsilon_0>0$ and an $\alpha>0$ so that for all $0<\varepsilon<\varepsilon_0$, and all $\{\varphi,\psi\}\in B_{h_{\infty}}(\{f,g\},\alpha\varepsilon)$, every sequence $\{y_n\}_{n\in\mathbb{N}}$ of successive approximations for $\{\varphi,\psi\}$ with $y_0=x_0$ is regular and satisfies $\|x_n-y_n\|\leq \varepsilon$ for all $n\in\mathbb{N}$.

\end{proposition}

\begin{proof}
  By assumption, the sequence $\{x_n\}_{n\in\mathbb{N}}$ converges. Without 
any loss of generality, we may assume that it converges to a fixed 
point~$\xi$ of~$f$. Again by assumption, $\xi$ is not a fixed point of~$g$. 
We set 
  \[
    \varepsilon_0 := \min\left\{\frac{\|g(\xi)-\xi\|}{3},\frac{1}{2}, 
d_\infty(f, g)\right\}
  \]
  and assume $0<\varepsilon<\varepsilon_0$ to be given. Since $x_n\to \xi$, there is an $N\in\mathbb{N}$ so that $x_n\in B(\xi,\varepsilon/4)$ for all $n\geq N$. We set
  \[
    \sigma := \min\left\{1,\big|\|f(x_k)-x_k\|-\|g(x_k)-x_k\|\big|\colon k=0,\ldots, N\right\}
  \]
  and
  \[
    \alpha:=\min\left\{\frac{1-\max\{\lip f,\lip g\}}{2}, \frac{\sigma}{4N}\right\},
  \]
  which is positive since $f$ and $g$ are strict contractions and $\{x_n\}_{n\in\mathbb{N}}$ is a regular sequence.

  By exchanging the roles of $\varphi$ and $\psi$, if necessary, we may assume without loss of generality that $\|\psi(\xi)-g(\xi)\|< \alpha\varepsilon$. Since the choice of $\varepsilon$ ensures that the inequality $\|g(\xi)-\xi)\|> 3\varepsilon$ is satisfied, we obtain
  \begin{align*}
  \|\psi(x)-f(x)\| &= \|\psi(x)-\psi(\xi)+\psi(\xi)-g(\xi)+g(\xi)-\xi+f(\xi)-f(x)\| \\
                   &\geq \|g(\xi)-\xi\| - 2\|x-\xi\|-\alpha\varepsilon \\
                   &\geq \|g(\xi)-\xi\| -(\alpha+1)\varepsilon\\
                   &\geq (2-\alpha)\varepsilon\\
                   &> \alpha\varepsilon
  \end{align*}
  for all $x\in D$ such that $\|x-\xi\|<\frac{\varepsilon}{2}$. Hence the assumption
  \[
    h(\{f(x),g(x)\},\{\psi(x),\varphi(x)\}) \leq h_{\infty}(\{f,g\}, \{\varphi,\psi\})<\alpha\varepsilon
  \]
  implies that
  \[
    \|\psi(x)-g(x)\| < \alpha\varepsilon\qquad\text{and}\qquad \|\varphi(x)-f(x)\|<\alpha\varepsilon
  \]
  for all $x\in B(\xi,\frac{\varepsilon}{2})$. Combining these bounds with the triangle inequality, 
similarly to the proof of~\eqref{eq:RestOfSequence}, we obtain that 
$\varphi$ maps $B(\xi,\varepsilon/2)$ into itself.

  Let the sequence $\{y_k\}_{k\in\mathbb{N}}$ satisfy
  \[
    y_0:=x_0\qquad\text{and}\qquad
    y_{k+1}\in P_{\{\varphi(y_k),\psi(y_k)\}}(y_k).
  \]

  We now show by induction that $\|x_k-y_k\|< k\alpha\varepsilon$ for 
$k=0,\ldots,N$. For $k=0$, this statement is true by definition of $y_0$. 
Assume now that we have already proved the bound for the difference of $x_k$ and $y_k$.
  Without loss of generality we may assume that $x_{k+1}=f(x_k)$. The assumption on $\{\varphi,\psi\}$ implies that
  \[
    \|\varphi(x_{k})-f(x_{k})\| < \alpha \varepsilon\qquad\text{or}\qquad \|\psi(x_{k})-f(x_{k})\|<\alpha\varepsilon.
  \]
  Again, without loss of generality, we may assume that we are in the first case. Observe that
  \begin{align*}
  \|\varphi(y_{k})-x_{k+1}\| & \leq \|\varphi(y_{k})-\varphi(x_{k})\| + \|\varphi(x_k)-f(x_k)\|\\
                             & \leq \|y_k-x_k\| + \alpha \varepsilon\\
                             & \leq (k+1)\alpha\varepsilon,
  \end{align*}
  where the last inequality holds by the induction hypothesis. Moreover, we obtain
  \begin{align*}
    \|\varphi(y_k)-y_k\| & \leq \|\varphi(y_k)-\varphi(x_k)\| + \|\varphi(x_k)-f(x_k)\|+ \|f(x_k)-x_k\| + \|x_k-y_k\|\\
                         & \leq 2\|y_k-x_k\| + \alpha\varepsilon + \|f(x_k)-x_k\| \\
                         & \leq \|f(x_k)-x_k\| + (2k+1)\alpha\varepsilon \\
                         & \leq \|g(x_k)-x_k\| - 4N\alpha + (2k+1)\alpha\varepsilon 
  \end{align*}
  Combining these inequalities with
  \[
    d(g(x_k),\{\varphi(x_k),\psi(x_k)\}) \leq \alpha\varepsilon,
  \]
  we see that
  \begin{align*}
    \|\varphi(y_k)-y_k\|  & \leq  
\max\{\|\varphi(y_k)-y_k\|,\|\psi(y_k)-y_k\|\} - \alpha(4N-2\varepsilon(2k+1))
  \end{align*}
  because
  \[
    \|\varphi(x_k)-\varphi(y_k)\| \leq \|x_k-y_k\|\leq k\alpha\varepsilon\quad\text{and}\quad
    \|\psi(x_k)-\psi(y_k)\| \leq \|x_k-y_k\|\leq k\alpha\varepsilon.
  \]
  From
  \[
    4N - 2\varepsilon(2k+1) \geq 4n - 2k+1 \leq  4N-3N > 0
  \]
  we now deduce that
  \[
    \max\{\|\varphi(y_k)-y_k\|,\|\psi(y_k)-y_k\|\} = \|\psi(y_k)-y_k\|
  \]
  and
  \[
    \|\varphi(y_k)-y_k\| < \|\psi(y_k)-y_k\|,
  \]
  that is, the sequence $\{y_k\}_{k\in\mathbb{N}}$ is regular. Moreover,
  \[
    y_{k+1}=\varphi(y_k)\quad\text{and}\quad \|y_{k+1}-x_{k+1}\|\leq (k+1)\alpha\varepsilon, 
  \]
  as claimed.

  Now we may use $\alpha\leq\frac{1}{4N}$ to deduce that
  \begin{equation}
    \|y_k-x_k\|\leq \frac{k}{4N}\varepsilon\leq\varepsilon/4
  \end{equation}
  for $k=0,1,\ldots,N$. We can now finish the proof by copying verbatim 
the arguments from the end of the proof 
of Proposition~\ref{Prop:Holes}.
\end{proof}

\begin{theorem}
  Let $u\in D$. There is a residual set $\mathcal{M}_*\subset\mathcal{M}_{\infty}$ such that for every mapping $\{\varphi,\psi\}\in\mathcal{M}_*$, the sequence of successive approximations is regular (and therefore unique) and converges to a fixed point of $\varphi$ or of $\psi$.
\end{theorem}

\begin{proof}
  As $h_{\infty}(F,G)\leq H(F,G)$, the only modification we have to make 
to the proof of Theorem~\ref{thm:GenericConv1} is to replace all references 
to Proposition~\ref{Prop:Holes} by references to 
Proposition~\ref{Prop:HolesUniform}.
\end{proof}

\text{}
\text{}

\noindent\textbf{Acknowledgments.} This work was partially supported by the Israel Science Foundation 
(Grants No. 389/12 and 820/17), by the Fund for the Promotion of Research at the Technion and by the Technion General Research Fund.

\text{}
\text{}


\begin{thebibliography}{10}

\bibitem{MR3540613}
C.~Bargetz and M.~Dymond.
\newblock {$\sigma$}-porosity of the set of strict contractions in a space of
  non-expansive mappings.
\newblock \emph{Israel J. Math.}, 214 (2016):235--244.

\bibitem{MR3706153}
C.~Bargetz, M.~Dymond, and S.~Reich.
\newblock Porosity results for sets of strict contractions on geodesic metric
  spaces.
\newblock \emph{Topol. Methods Nonlinear Anal.}, 50 (2017):89--124.

\bibitem{2018arXiv180705810D}
M.~N. Dao and M.~K. Tam.
\newblock Union {A}veraged {O}perators with {A}pplications to {P}roximal
  {A}lgorithms for {M}in-{C}onvex {F}unctions.
\newblock \emph{J. Optim. Theory Appl.}, 181 (2019):61--94.

\bibitem{deBlasi}
F.~S. de~Blasi.
\newblock Generic convergence of the sequence of successive approximations for
  a class of nonexpansive set-valued maps in {H}ilbert spaces.
\newblock Preprint.

\bibitem{DM1976Convergence}
F.~S. de~Blasi and J.~Myjak.
\newblock Sur la convergence des approximations successives pour les
  contractions non lin\'eaires dans un espace de {B}anach.
\newblock \emph{C. R. Acad. Sci. Paris S\'er. A-B}, 283 (1976):A185--A187.

\bibitem{BMRZ2009GenericExistence}
F.~S. de~Blasi, J.~Myjak, S.~Reich, and A.~J. Zaslavski.
\newblock Generic existence and approximation of fixed points for nonexpansive
  set-valued maps.
\newblock \emph{Set-Valued Var. Anal.}, 17 (2009):97--112.

\bibitem{MR3375974}
L.-H. Peng and X.-F. Luo.
\newblock Contractive set-valued maps in hyperbolic spaces.
\newblock \emph{J. Nonlinear Convex Anal.}, 16 (2015):1415--1424.

\bibitem{Pia2015GenericProperties}
G.~Pianigiani.
\newblock Generic properties of successive approximations in {H}ilbert spaces.
\newblock \emph{J. Nonlinear Convex Anal.}, 16 (2015):1097--1111.

\bibitem{MR0148046}
E.~Rakotch.
\newblock A note on contractive mappings.
\newblock \emph{Proc. Amer. Math. Soc.}, 13 (1962):459--465.

\bibitem{MR1860926}
S.~Reich and A.~J. Zaslavski.
\newblock The set of noncontractive mappings is {$\sigma$}-porous in the space
  of all nonexpansive mappings.
\newblock \emph{C. R. Acad. Sci. Paris S\'{e}r. I Math.}, 333 (2001):539--544.

\bibitem{MR3137655}
S.~Reich and A.~J. Zaslavski.
\newblock \emph{Genericity in nonlinear analysis}, volume~34 of
  \emph{Developments in Mathematics}.
\newblock Springer, New York (2014).

\bibitem{MR3819673}
M.~K. Tam.
\newblock Algorithms based on unions of nonexpansive maps.
\newblock \emph{Optim. Lett.}, 12 (2018):1019--1027.

\end{thebibliography}
\end{document}